\documentclass[11pt,oneside,reqno]{amsart}

\advance\topmargin by +0.1in \advance\textheight by +0.2in
\advance\oddsidemargin by -0.5in \advance\textwidth by +1.1in

\usepackage{amssymb,amsmath,amscd,amsthm}
\usepackage{hyperref}
\usepackage{empheq}
\usepackage{cases}
\usepackage{color}
\usepackage{enumerate}

\DeclareMathOperator{\Od}{Od}

\newcommand{\df}{\displaystyle\frac}
\newcommand{\ds}{\displaystyle\sum}

\newcommand{\CP}{{\mathbb {CP}}}

\newcommand{\HP}{{\mathbb {HP}}}
\newcommand{\OP}{{\mathbb {OP}}}
\newcommand{\QP}{{\mathbb {QP}}}
\newcommand{\F}{{\mathbb F}}
\newcommand{\Q}{{\mathbb Q}}

\newcommand{\Z}{{\mathbb Z}}

\newcommand{\LP}{{\mathcal{L}}}
\newcommand{\AP}{{\hat{A}}}

\DeclareMathOperator{\wt}{wt}
\newcommand{\A}{{\mathcal{A}}}

\makeatletter
\newtheorem*{rep@theorem}{\rep@title}
\newcommand{\newreptheorem}[2]{%
\newenvironment{rep#1}[1]{%
 \def\rep@title{#2 \ref{##1}}%
 \begin{rep@theorem}}%
 {\end{rep@theorem}}}
\makeatother

\newtheorem{theorem}{Theorem}

\newtheorem{maintheorem}{Theorem}

\newreptheorem{maintheorem}{Theorem}

\newreptheorem{maincorollary}{Corollary}

\newtheorem*{nonumbertheorem}{Theorem}

\makeatletter
\@addtoreset{theorem}{chapter}
\makeatother
\newtheorem{proposition}[theorem]{Proposition}

\newtheorem{corollary}[theorem]{Corollary}
\newtheorem{lemma}[theorem]{Lemma}
\theoremstyle{definition}

\newtheorem{remark}[theorem]{Remark}

\title{On dimensions supporting a rational projective plane}
\author{Lee Kennard}\address{Department of Mathematics, University of Oklahoma, Norman, OK 73019}\email{kennard@ou.edu}\urladdr{www.math.ou.edu/~kennard}
\author{Zhixu Su}\address{Department of Mathematics, Indiana University, Bloomington, IN 47408}\email{zhisu@indiana.edu}\urladdr{}
\date{\today}

\begin{document}

\begin{abstract}
A rational projective plane ($\mathbb{QP}^2$) is a simply connected, smooth, closed manifold $M$ such that $H^*(M;\mathbb{Q}) \cong \mathbb{Q}[\alpha]/\langle \alpha^3 \rangle$. An open problem is to classify the dimensions at which such a manifold exists. The Barge--Sullivan rational surgery realization theorem provides necessary and sufficient conditions that include the Hattori--Stong integrality conditions on the Pontryagin numbers. In this article, we simplify these conditions and combine them with the signature equation to give a single quadratic residue equation that determines whether a given dimension supports a $\mathbb{QP}^2$. We then confirm existence of a $\mathbb{QP}^2$ in two new dimensions and prove several non-existence results using factorization of the numerators of the divided Bernoulli numbers. We also resolve the existence question in the Spin case, and we discuss existence results for the more general class of  rational projective spaces. 
\end{abstract}

\maketitle

The rank one symmetric spaces given by the complex projective plane $\CP^2$, the quaternionic projective plane $\HP^2$, and the Cayley plane $\OP^2$ have the property of being simply connected, closed, smooth manifolds $M$ with cohomology ring isomorphic to $\Z[\alpha]/\langle \alpha^3\rangle$. These examples exist in dimensions $4$, $8$, and $16$ respectively. By Adams' resolution of the Hopf invariant one problem, no other dimension supports such a manifold (see \cite{Adams60}). In fact, Adams' proof also covers mod $2$ projective planes, i.e., manifolds as above with the property that $H^*(M;\Z_2) \cong \Z_2[\alpha]/\langle \alpha^3\rangle$ (cf. \cite{Liulevicius62,ShimadaYamanoshita61} for work on odd prime analogues). 

By analogy, a rational projective plane (denoted $\QP^2$) is a simply connected, closed, smooth manifold $M$ with rational cohomology ring isomorphic to $\Q[\alpha]/\langle \alpha^3\rangle$. The dimensions at which such a manifold exist are not yet classified. The second author began work on this question and proved that a $\QP^2$ exists in dimension $32$ but not in any other dimension less than $32$ aside from $4$, $8$, and $16$ (see \cite{Su14}). It was also shown that a necessary condition for the existence of a $\QP^2$ in dimensions $n > 4$ is that $n = 8k$ for some $k$. Fowler and the second author showed further that $k$ must be of the form $2^a + 2^b$, and that no $\QP^2$ exists in dimensions $32 < n < 128$ and $128 < n<256$. 

\begin{maintheorem}\label{thm:UpTo512}
A $\QP^2$ exists in dimension $n \leq 512$ if and only if $n \in \{4,8,16,32,128,256\}$. Moreover, no $\QP^2$ exists in any dimension $512 < n < 2^{13}$, except for five possible exceptions, $n \in \{544, 1024, 2048, 4160, 4352\}$.
\end{maintheorem}

The approach is by rational surgery, as in \cite{Su14,FowlerSu16}. By the rational surgery realization theorem of Barge and Sullivan, the existence of a $\QP^2$ in a particular dimension is equivalent to the existence of formal Pontryagin classes that satisfy the Hirzebruch signature equation and the Hattori--Stong integrality conditions (see \cite{Su14}, cf. \cite{Barge76,Sullivan77}). The main step in this article is to show that two of the Hattori--Stong integrality conditions are sufficient to imply the others (see Section \ref{sec:ReducingIntegralityConditions}). We then prove that the signature equation and these two integrality conditions are equivalent to a single quadratic residue equation (see Section \ref{sec:ReducingSingleEquation}). In Sections \ref{sec:Existence} and \ref{sec:Non-existence}, we apply these simplifications to answer the existence question for $\QP^2$ in all but five dimensions less than $2^{13}$ and, in particular, to prove Theorem \ref{thm:UpTo512}.

Section \ref{sec:Non-existence} also contains some general non-existence results. They rely on congruences of Carlitz and Kummer, and obstructions from irregular prime factors of the numerators of the divided Bernoulli numbers. The results provide new infinite families of dimensions that do not support a $\QP^2$ (see Section \ref{sec:Non-existence} for precise families of dimensions obstructed).

\begin{maintheorem}\label{thm:InfinitelyManyPowersOfTwo}
There are infinitely many dimensions of the form $2^a$, and infinitely many dimensions of the form $8(2^a + 2^b)$ with $a\not =b$, that do not support the existence of a $\QP^2$.
\end{maintheorem}

As mentioned above, Fowler and the second author proved that no $\QP^2$ exists in a dimension not of the form $2^a$ or $8(2^a + 2^b)$ for some $a\not =b$ (see \cite{FowlerSu16}). It remains an open question whether infinitely many dimensions, and whether any dimension of the latter form, can support a rational projective plane. 

We specialize in Section \ref{sec:Spin} to the Spin case, and we classify the dimensions that support a Spin $\QP^2$. Note that $\HP^2$ and $\OP^2$ are examples in dimensions $8$ and $16$.

\begin{maintheorem}\label{thm:Spin}
A Spin $\QP^2$ exists in dimension $n$ if and only if $n \in \{8,16\}$.
\end{maintheorem}

The necessary and sufficient conditions for existence of Spin $\QP^2$ are analogous to the smooth case except the Spin Hattori--Stong integrality conditions involve the $\hat A$--genus instead of $\LP$--genus. The obstructions coming from the signature equation, integrality of $\hat A$--genus, and one of the Spin Hattori--Stong conditions are sufficient to prove no Spin $\QP^2$ exists in dimensions $n > 16$.

Finally, we discuss existence questions for rational projective spaces. Extending the notation above, let $\QP^n_d$ denote a simply connected, smooth, closed manifold in dimension $dn$ with rational cohomology isomorphic to $\Q[\alpha]/\langle\alpha^{n+1}\rangle$, $|\alpha|=d$. For example, a $\QP^2_8$ is a rational Cayley plane. The main existence result we prove in Section \ref{sec: projective-space} is the following.

\begin{maintheorem}\label{thm:QPn}
If a $\QP^2_{4k}$ exists, then a $\QP^{2m}_{4k/m}$ exists whenever $4k/m \in 2\Z$.
\end{maintheorem}
We illustrate this theorem with some examples:
	\begin{enumerate}[{I.}]
	\item By Theorems \ref{thm:UpTo512} and \ref{thm:QPn}, higher dimensional analogues $\QP^n_8$ of rational Cayley planes exist for $n \in \{4,16,32\}$. Note that $\QP^n_8$ exist for all odd $n$ (see \cite{FowlerSu16}).
	\item No $\QP^2_{32}$ exists, however there exist higher dimensional analogues, $\QP_{32}^4$ and $\QP_{32}^8$.
	\end{enumerate}
In light of the last example, it may be asked whether every power of two can be realized as the degree $d$ of a $\QP^n_d$ for some $n \geq 2$. The answer is yes by Theorem \ref{thm:QPn} if infinitely many dimensions equal to a power of two support a $\QP^2$, but this too remains an open question.

\subsection*{Acknowledgements}
We want to thank Matthias Kreck and Don Zagier for email communication regarding their independent work on this problem, which includes results equivalent to Theorem A stated above, as well as nonexistence results beyond the range of dimensions we considered in this paper. 

We also want to thank Yang Su and Jim Davis for communication about this problem and Sam Wagstaff for discussions on Bernoulli numbers that made possible the proof of Theorem \ref{thm:InfinitelyManyPowersOfTwo}. Finally we are grateful to the referee for carefully reading and making suggestions to improve the paper. The first author was supported by NSF Grant DMS 1622541.

\section{Preliminaries}

We consider the question of whether a $\QP^2$ exists in dimension $n$. By the graded commutativity of the cup product, the dimension $n$ must be a multiple of four. Moreover, the second author proved that, except for dimension four, a $\QP^2$ can exist only if $n = 8k$ for some integer $k$ (see \cite{Su14}).

We first outline the necessary and sufficient condition for the existence of a simply connected, closed, smooth manifold realizing a prescribed rational cohomology ring. If $M^{8k}$ is an $8k$--dimensional $\QP^2$, then all its rational Pontryagin classes vanish except for $p_k \in H^{4k}(M;\Q)$ and $p_{2k} \in H^{8k}(M;\Q)$. Hence the total $\LP$ class can be written as
$$\LP=1+s_kp_k+s_{k,k}p_k^2+s_{2k}p_{2k}.$$
As derived in \cite{MilnorStasheff} and \cite{Anderson69}, the coefficients are
	\begin{eqnarray*}
	s_k&=&\df{2^{2k}(2^{2k-1}-1)|B_{2k}|}{(2k)!},\\
	s_{k,k}&=&\frac{1}{2}(s_{k}^2-s_{2k}).
	\end{eqnarray*}
With a choice of orientation, we may assume that the signature of $M$ is $1$. The following necessary conditions must hold true:
\begin{enumerate}
\item
(Hirzebruch signature equation)\ \  \begin{equation}\langle \LP(p_k,p_{2k}),\mu\rangle=s_{k,k}\langle p_k^2,\mu\rangle+s_{2k}\langle p_{2k},\mu\rangle=1,\end{equation}
\item
(Hattori--Stong integrality condition from $\Omega^{SO}_{8k}$)\
\begin{equation}\langle\Z[e_1,\,e_2, \ldots]\mathord{\mathord{\cdot}}\LP\,,\,\mu\rangle\in\Z[1/2]\end{equation}
\item
(Pontryagin numbers of $\QP^2$)
\begin{equation}\langle p_k^2,\mu\rangle= x^2 \mbox{ and } \langle p_{2k},\mu\rangle=y \ \ \mbox{ for some integers } x \mbox{ and } y
\end{equation}
\end{enumerate} 

Condition (3) is a consequence of the rational cohomology ring structure of $M$. Since $H^*(M;\Q)=\Q[\alpha]/\langle \alpha^3\rangle$, where $\alpha$ is any generator in $H^{4k}(M;\Q)$, we may write the Pontryagin classes $p_k=a\alpha$ and $p_{2k}=b\alpha^2$ for some rational numbers $a$ and $b$. By the choice of orientation, the rational intersection form of $M$ is isomorphic to $\langle 1 \rangle$ and the signature is 1, so we must have $\langle \alpha^2, \mu\rangle= r^2$ for some rational number $r$, then the Pontryagin numbers of $M$ can be expressed as $\langle p_k^2,\mu\rangle=a^2r^2= x^2$ and $\langle p_{2k},\mu\rangle=br^2= y$, where $x$ and $y$ must be integers because the Pontryagin numbers of a smooth manifold must be integers. With this substitution, the signature equation (1) can be written as 
$$s_{k,k}x^2+s_{2k}y= 1.$$

The Hattori--Stong integrality condition (2) characterizes the integral lattice in $\Q^{p(8k)}$ formed by all possible Pontryagin numbers of a smooth $8k$-dimensional manifold in $\Omega^{SO}_{8k}$. The $e_l$ classes are defined as follows. If one writes the total Pontryagin class formally as $p=\displaystyle\prod_i(1+x_i^2)=\prod_i(1+t_i)$, the $k$-th Pontryagin class can be expressed as the $k$-th elementary symmetric function of $t_i$.
 $$p_k=\sigma_k(t)=\ds_{i_1<\cdots<i_k} t_{i_1}t_{i_2}\cdots t_{i_k}.$$

Consider the variable $T_i$ that is written as a power series of $t_i$ as follows:$$T_i:=e^{\sqrt{t_i}}+e^{-\sqrt{t_i}}-2=\ds_{n=1}^{\infty} \df{2t_i^n}{(2n)!}=2\left(\df{t_i}{2!}+\df{t_i^2}{4!}+\ldots\right).$$
We denote the $l$-th elementary symmetric functions of the variable $T_i$ as 
$$e_{l}:=\sigma_l(T)=\ds_{i_1<\cdots<i_l} T_{i_1}T_{i_2}\cdots T_{i_l}.$$
Since each $e_l$ class can be written as a rational linear combination of monomials of the Pontryagin classes $p_k$, in our case of $\QP^2$, each $e_l$ class can be written as a rational linear combination of $p_k^2$ and $p_{2k}$. Therefore the Hattori--Stong Integrality condition (2) can be expressed as a set of integrality conditions on the Pontryagin numbers $\langle p_k^2,\mu\rangle=x^2$ and $\langle p_{2k},\mu\rangle=y$. 

As discussed in \cite{Su14}, by the rational surgery realization theorem (\cite{Barge76} and \cite{Sullivan77}), the above necessary conditions are also the sufficient conditions for the existence of a $\QP^2$. More precisely, there exists a smooth closed manifold $M$ in dimension $n=8k$ such that $H^*(M;\Q)=\Q[\alpha]/\langle \alpha^3 \rangle$ if and only if there exist pair of integers $x^2$ and $y$ which realize the Pontraygin numbers of a $\QP^2$ as in (3), and they satisfy the signature equation (1) and the Hattori--Stong integrality conditions in (2).  So the problem is reduced to solving a system of Diophantine equations, which is purely an elementary number theoretic problem.

\section {Reducing the integrality conditions}\label{sec:ReducingIntegralityConditions}

In the proof of existence of $32$-dimensional $\QP^2$ in \cite{Su14}, the second author explicitly computed the Hattori--Stong integrality condition in dimension $32$. The calculation involved concretely writing each $e_l$ classes in Condition (2) in terms of the Pontraygin classes $p_4^2$ and $p_8$. In this section, we simplify the Hattori--Stong integrality condition in  our case of $\QP^2$ to a much simpler form. The argument works for any dimension. 

\begin{theorem}\label{SOiff}
There exists a $\QP^2$ in dimension $n=8k$ if and only if there are integers $x$ and $y$ that satisfy the following conditions:
\begin{subnumcases}
\ \  s_{k,k}x^2+s_{2k}y= 1 \label{signature}\\
 \left(\df{(-1)^{k+1}s_k}{(2k-1)!}+\df{1}{2(4k-1)!}\right)x^2 -\df{y}{(4k-1)!}\in\Z[1/2] \label{e1L}\\
 \df{x^2}{[(2k-1)!]^2}\in\Z[1/2] \label{e1e1L}
\end{subnumcases}
Moreover, for any pair of integers $x$ and $y$ satisfying the above conditions, there is a $\QP^2$ whose Pontryagin numbers satisfy $\langle p_k^2,\mu\rangle=x^2$ and $\langle p_{2k},\mu\rangle=y$.
\end{theorem}

We spend the rest of this section on the proof. Condition (1), the signature equation, is the same as Equation \eqref{signature}, and Condition (3) on the integrality of the Pontryagin numbers is implicit in the statement. Therefore it is sufficient to show that the Hattori--Stong integrality conditions stated in Condition (2) are equivalent to Equations (\ref{e1L}) and (\ref{e1e1L}). Since a $\QP^2$ satisfies $p_\omega=0$ except possibly for $p_k$, $p_k^2$ and $p_{2k}$, Condition (2) is equivalent to the claim that $\langle e_l \mathord{\cdot}\LP, \mu\rangle \in \Z[1/2]$ for all $1 \leq l \leq 2k$ and that $\langle e_le_m\mathord{\cdot}\LP,\mu\rangle\in\Z[1/2]$ whenver $1 \leq l + m \leq 2k$.

In the following lemma, we calculate the $e_l$ class in terms of the Pontryagin classes.

\begin{lemma}\label{el}
If $p_\omega=0$ except $p_k$, $p_k^2$ and $p_{2k}$, then
\begin{equation}\label{e1a}
e_1=\df{(-1)^{k+1}}{(2k-1)!}\,p_k + \df{1}{2(4k-1)!}\,p_k^2+\df{-1}{(4k-1)!}\,p_{2k}
\end{equation}
and
\begin{subequations}
\begin{eqnarray}
e_l&=&\df{(-1)^{l+1}}{l}\left[M_l(2k)e_1+[M_l(k)-M_l(2k)]\df{(-1)^{k+1}}{(2k-1)!}\,p_k\right]+\df{1}{2}\ds_{i=1}^{l-1}e_ie_{l-i} \label{el1}\\
&=&\df{(-1)^{k+l}M_l(k)}{l(2k-1)!}\,p_k +\df{(-1)^{l}M_l(2k)}{l(4k-1)!}\,p_{2k}+ p_k^2 \mbox{ term } \label{el2}
\end{eqnarray}
\end{subequations}
where $M_l(k)=\ds_{j=0}^{l-1}(-1)^j\binom{2l}{j}(l-j)^{2k}$.
\end{lemma}
\begin{proof}
For any partition $\omega=(\omega_1,  \cdots,\omega_r)$, there is the monomial symmetric polynomial $m_{\omega}(t)=\ds_{i_1<\cdots<i_{r}} t_{i_1}^{\omega_1}t_{i_2}^{\omega_2}\cdots t_{i_{r}}^{\omega_{r}}$. Let us denote the $m_l$ polynomial of the variable $T_i$ by
$$m_l:=m_{l}(T)=\ds_{i} T_{i}^l.$$
Note, in particular, that 
	\begin{eqnarray}\label{Ttot}
	m_1 = \sum_i T_i = \sum_{k=0}^\infty \frac{2}{(2k)!} \sum_i t_i^k = \sum_{k=0}^\infty \frac{2}{(2k)!} m_k(t).
	\end{eqnarray}
Similar to the calculation carried out in \cite{Barge74} page 488, we find the coefficient of $p_k$ and $p_k^2$ in $m_{l}$. Let $\{-\}_{k}$ denote the degree $k$ terms in an expression. We have
 \begin{eqnarray}\label{m}
\{m_{l}\}_k=\left\{\ds_{i}\left(e^{\sqrt{t_i}}+e^{-\sqrt{t_i}}-2\right)^l\right\}_k&=&\left\{\ds_{i}\left(e^{\sqrt{t_i}/2}-e^{-\sqrt{t_i}/2}\right)^{2l}\right\}_k\nonumber\\
	&=&\left\{\ds_{i}\ds_{j=0}^{2l}(-1)^j\binom{2l}{j}e^{\sqrt{t_i}(l-j)}\right\}_k\nonumber\\
	&=&\ds_{i}\ds_{j=0}^{2l}(-1)^j\binom{2l}{j}\df{t_i^{k}(l-j)^{2k}}{(2k)!}\nonumber\\
	&=&\frac{m_k(t)}{(2k)!} \sum_{j=0}^{2l} (-1)^j \binom{2l}{j} (l-j)^{2k}\nonumber\\
	&=&\frac{2}{(2k)!} m_k(t) M_l(k)\nonumber
\end{eqnarray}
Using Equation \eqref{Ttot} and the fact that $m_l$ only contains terms of degree at least $l$, this implies that
$$\{m_{l}\}_k=\begin{cases}
    M_l(k)\{m_1\}_k=M_l(k)\{e_1\}_k& \text{if } l\leq k\\
    0             &  \text{if } l> k
\end{cases}
$$
By the Newton-Girard identities %(\cite{MilnorStasheff} page 195, Problem 16-A) 
relating the monomial symmetric function $m_k(t)$ with the elementary symmetric functions $p_i=\sigma_i(t)$, 

\begin{eqnarray}e_1=m_1=\{m_1\}_k+\{m_1\}_{2k}&=&\df{2}{(2k)!}m_k(t)+\df{2}{(4k)!}m_{2k}(t)\nonumber\\
&=&\df{2}{(2k)!}(-1)^{k+1}k\,p_k+\df{2}{(4k)!}(k\,p_k^2-2k\,p_{2k})\nonumber\\
&=&\df{(-1)^{k+1}}{(2k-1)!}\,p_k+\df{1}{2(4k-1)!}\,p_k^2+\df{-1}{(4k-1)!}\,p_{2k} \label{e1b}
\end{eqnarray}
\begin{eqnarray}m_l=\{m_l\}_k + \{m_l\}_{2k}	&=& M_l(k)\{e_1\}_k + M_l(2k)\{e_1\}_{2k}\nonumber\\
&=&M_l(2k)\left(\{e_1\}_k+\{e_1\}_{2k}\right) + \left[M_l(k)-M_l(2k)\right]\{e_1\}_k\nonumber\\
&=&M_l(2k)e_1+[M_l(k)-M_l(2k)]\df{(-1)^{k+1}}{(2k-1)!}\,p_k \label{sl}
\end{eqnarray}

Again by the Newton-Girard identities relating the symmetric functions $m_l=m_l(T)$ and $e_i=\sigma_i(T)$,
\begin{equation*}\label{girard}m_{l}=(-1)^{l+1}l\, e_l+(-1)^{l+2}\df{l}{2}\ds_{i=1}^{l-1}e_ie_{l-i}+\ds_{\ell(\omega)>2}c_{\omega} e_{\omega}.\end{equation*}
Since $p_k$, $p_k^2$ and $p_{2k}$ are the only non-trivial classes, and each class $e_{i}$ can be expressed as a rational linear combination of these classes, $e_\omega=0$ if the partition $\omega$ has length $\ell(\omega)>2$. Then we may express
\begin{eqnarray}e_l&=&\df{(-1)^{l+1}}{l}m_l+\df{1}{2}\ds_{i=1}^{l-1}e_ie_{l-i},\nonumber\end{eqnarray}
which gives \eqref{el1} if we plugin \eqref{sl}, and \eqref{el2} if we plugin \eqref{e1b}.
\end{proof}

Using Formula \eqref{e1a} for $e_1$ from this lemma, we obtain the formulas
$$
\left\{ \begin{array}{rll}
\langle e_1\mathord{\cdot}\LP,\mu\rangle&=\left(\df{(-1)^{k+1}s_k}{(2k-1)!}+\df{1}{2(4k-1)!}\right)x^2 -\df{y}{(4k-1)!}, \\
\langle e_1e_1\mathord{\cdot}\LP,\mu\rangle&= \df{x^2}{[(2k-1)!]^2},\\
 \end{array} \right.
$$
where $\langle p_k^2,\mu\rangle=x^2$ and $\langle p_{2k},\mu\rangle=y$. Note that these are Conditions \eqref{e1L} and \eqref{e1e1L} in Theorem \ref{SOiff}. To complete the proof of Theorem \ref{SOiff}, it suffices to prove that the conditions $\langle e_1\mathord{\cdot}\LP,\mu\rangle \in \Z[1/2]$ and $\langle e_1e_1\mathord{\cdot}\LP,\mu\rangle \in \Z[1/2]$ imply that $\langle e_l\mathord{\cdot}\LP,\mu\rangle \in \Z[1/2]$ for all $l \leq k$ and $\langle e_le_m\mathord{\cdot}\LP,\mu\rangle \in \Z[1/2]$ for all $l + m \leq 2k$.

\begin{lemma}\label{ee}
If $p_{\omega}=0$ except possibly for $p_k$, $p_k^2$, and $p_{2k}$, and if $\langle e_1e_1\mathord{\cdot}\LP,\mu\rangle\in\Z[1/2]$, then $\langle e_le_m\mathord{\cdot}\LP,\mu\rangle\in\Z[1/2]$ for all $l,m\geq 1$.
\end{lemma}
 \begin{proof}

By Equation \eqref{e1a} in Lemma \ref{el}, $e_1e_1\mathord{\cdot}\LP= \left(p_k / (2k-1)! \right)^2$. Together with Equation \eqref{el2}, this implies that 
$$e_le_m\mathord{\cdot}\LP=\df{(-1)^{l}M_l(k)}{l(2k-1)!}\df{(-1)^mM_m(k)}{m(2k-1)!}p_k^2=(-1)^{l+m}\df{M_l(k)}{l}\df{M_m(k)}{m}e_1e_1\mathord{\cdot}\LP.$$ 
To prove the lemma, it is sufficient to show that $l$ divides $M_l(k)$ for any integer $l$. By the definition of $M_l(k)$ in Lemma \ref{el}, it suffices to show that $l$ divides $\binom{2l}{j}(l-j)$ for all $0 \leq j \leq l-1$. To see this, we use the fact that $a/\gcd(a,b)$ divides $\binom{a}{b}$. In particular, $2l$ divides $\binom{2l}{j}\mbox{gcd}(2l,j)$, which in turn divides $2\binom{2l}{j}\gcd(l,j)$. Hence $l$ divides $\binom{2l}{j}(l-j)$, as required.
\end{proof}

Together with Lemma \ref{ee} and the comments preceding it, the following lemma implies Theorem \ref{SOiff}.

\begin{lemma}\label{e}
If $p_{\omega}=0$ except possibly for $p_k$, $p_k^2$, and $p_{2k}$, and if $\langle e_1\mathord{\cdot}\LP,\mu\rangle \in \Z[1/2]$ and $\langle e_1e_1\mathord{\cdot}\LP,\mu\rangle \in \Z[1/2]$, then $\langle e_l\mathord{\cdot}\LP,\mu\rangle\in\Z[1/2]$ for all $l \geq 1$.
\end{lemma}

\begin{proof}
By lemma \ref{el} equation \eqref{el1}, 
\begin{equation*}\langle e_l\mathord{\cdot}\LP,\mu\rangle=(-1)^{l+1}{\df{M_l(2k)}{l}\langle e_1\mathord{\cdot}\LP,\mu\rangle}
+(-1)^{l+k}{\df{M_l(k)-M_l(2k)}{l(2k-1)!}\langle p_k\mathord{\cdot}\LP,\mu\rangle}
+\df{1}{2}\ds_{i=1}^{l-1}{\langle e_ie_{l-i}\mathord{\cdot}\LP,\mu\rangle}
\end{equation*}

By the proof of Lemma \ref{ee}, we have that $l$ divides $M_l(k)$, so the assumption of the lemma implies that the first term lies in $\Z[1/2]$. Moreover, the terms involving $\langle e_i e_{l-i} \mathord{\cdot} \LP,\mu\rangle$ lie in $\Z[1/2]$ by Lemma \ref{ee}, so it suffices to show that the second term lies in $\Z[1/2]$. To do this, note that
\[\frac{\langle p_k\mathord{\cdot}\LP,\mu\rangle}{(2k-1)!}
	= s_k (2k-1)! \langle e_1e_1\mathord{\cdot}\LP,\mu\rangle
	= \frac{2^{2k}(2^{2k-1} - 1)|B_{2k}|}{2k} \langle e_1e_1\mathord{\cdot}\LP,\mu\rangle.\]
and that
\[\frac{M_l(2k) - M_l(k)}{l} = \sum_{j=0}^{l-1} (-1)^j \frac{1}{l}\binom{2l}{j} (l-j)^{2k} \left((l-j)^{2k}-1\right).\]
Hence it suffices to prove that $l$ divides $\binom{2l}{j}(l-j)$ and that $(l-j)^{2k-1}\left((l-j)^{2k} - 1\right)|B_{2k}| / (2k)$ is an integer for all $0 \leq j \leq l-1$. The first of these statements holds by the proof of Lemma \ref{ee}. The second nearly holds by the Lipschitz-Sylvester theorem that $a^{2k}(a^{2k}-1)B_{2k}/(2k) \in \Z$ for all integers $a$ (see, for example \cite[p. 247]{IrelandRosen90}). In fact, an elementary argument shows that the statement still holds with $a^{2k}$ replaced by $a^{2k-1}$, as required (cf. \cite{Slavutskii95} for a stronger statement that $a^{2k}$ can also be replaced by $a^{\lfloor \log_2(2k)\rfloor+1}$).
\end{proof}

\section{Reducing to a single quadratic residue equation}\label{sec:ReducingSingleEquation}

It was proved in \cite{FowlerSu16} that $\QP^2$ can only exist in dimensions of the form $n=8k$ where $k=2^a+2^b$ for some integers $a \leq b$. This result follows by a consideration of the $2$-adic order of the coefficients in the signature equation. Here we divide into two cases, $k=2^a$ and $k=2^a+2^b$ with $a < b$. In each case, we combine the integrality conditions involving \eqref{e1L} and \eqref{e1e1L} with the signature equation \eqref{signature}. The result is equivalent to one single quadratic residue equation.

We introduce the following notation and recall some well known facts about Bernoulli numbers (see \cite[Chapter 15]{IrelandRosen90}):
\begin{itemize}
\item
$\nu_2(n)$: the 2-adic order of $n$.
\item
$\wt(n)$: the number of ones in the binary expansion of $n$.
\item
$\Od[n]$: the odd part of $n$, i.e., $n/2^{\nu_2(n)}$.
\item
$N_{n}$: the numerator of the divided Bernoulli number $\frac{|B_{n}|}{n}$. $N_{n}$ is 1 only for $n = 2,4,6,8,10,14$, otherwise it is a product of powers of irregular primes.
\item
$D_{n}$: the denominator of the divided Bernoulli number $\frac{|B_{n}|}{n}$. By the theorem of von Staudt-Clausen, $D_{2k}=\prod\limits_{p-1\,|\, 2k}p^{\mu+1}$ where $p^{\mu}$ is the highest power of $p$ dividing $2k$. %(\cite{MilnorStasheff}, page 284)
\item
$OD_{n}$: the odd part of the denominator of the divided Bernoulli number $\frac{|B_{n}|}{n}$.\\
\end{itemize}

The $e_1^2\cdot\LP$ condition \eqref{e1e1L} in Theorem \ref{SOiff} requires $\frac{x^2}{[(2k-1)!]^2} \in \Z[1/2]$. It follows that 
$$x=\Od[(2k-1)!]\bar{x}$$
for some integer $\bar{x}$ with the same parity as $x$. Together with a change of variable $z=2y-x^2$, the signature equation \eqref{signature} can be written as:\begin{eqnarray}
s_k^2\,\left(\Od[(2k-1)!]\,\bar{x}\right)^2+s_{2k}\,z&=&2\label{signature-eq1},
\end{eqnarray}
where $\bar{x}$ and $z$ must have the same parity. So far, this shows that a $\QP^2$ exists in dimension $8k$ if and only if there exist $\bar{x},z \in \Z$ such that $\bar{x} \equiv z \bmod{2}$, Equation \eqref{signature-eq1}, and Equation \eqref{e1L} in Theorem \ref{SOiff} hold.

Next, we eliminate Equation \eqref{e1L} through another change of variables. Before proceeding, we need the following 2-adic numbers:
\begin{center}
$\nu_2\left(\frac{|B_{2k}|}{2k}\right)=-\nu_2(D_{2k})=-(\nu_2(k)+2),$\\
$\nu_2[(2k-1)!]=2k-\nu_2 (k)-\wt(k)-1$,\\
$\nu_2[(4k-1)!]=4k-\nu_2(k)-\wt(k)-2$.\\
\end{center}

Using the variables $\bar{x}$ and $z$, the $e_1\mathord{\mathord{\cdot}} \LP$ condition \eqref{e1L} 
can be written as
	\begin{eqnarray}
	(-1)^{k+1}\df{s_k}{(2k-1)!}\left(\Od[(2k-1)!]\bar{x}\right)^2-\df{1}{2(4k-1)!}z&\in&\Z[1/2].\label{e1L-eq1}
	\end{eqnarray}	
Since the $2$-adic order of the left hand side is
\[\inf\left\{\nu_2\left[\frac{s_k}{(2k-1)!}\right],\ \nu_2\left[\frac{1}{2(4k-1)!}\right]\right\}= \nu_2\left[\frac{1}{2(4k-1)!}\right]=-[4k-\nu_2(k)-\wt(k)-1],\]
we can multiply by $2^{4k-\nu_2(k)-\wt(k)-1}$ in \eqref{e1L-eq1} and expand $s_k$ using the definition to get
\begin{eqnarray}
	\frac{(-1)^{k+1}2^{2k+\wt(k)-1}\,(2^{2k-1}-1)N_{2k}}{OD_{2k}}\,\bar{x}^2-\df{1}{\Od[(4k-1)!]}z&\in&\Z.\label{e1L-eq2}
	\end{eqnarray}
This allows us to write $z$ as
\begin{equation}
z	=	\Od[(4k-1)!]\left(\frac{(-1)^{k+1}2^{2k+\wt(k)-1}\,(2^{2k-1}-1)N_{2k}}{OD_{2k}}\,\bar{x}^2+l\right)\label{e1L-eq3}
\end{equation}
for some integer $l$ with the same parity as  $z$. The following lemma ensures that $z \in \Z$ for any integers $\bar{x}$ and $l$ (of the same parity or not). This lemma implies that Equation \eqref{e1L-eq3} holds for some $l$ of the same parity at $z$ if and only if Condition \eqref{e1L-eq1} holds. In other words, we can use Equation \eqref{e1L-eq3} to make the change of variables from $(\bar{x},z)$ with $\bar{x} \equiv z \pmod{2}$ to $(\bar{x},l)$ with $\bar{x}\equiv l\pmod{2}$.

\begin{lemma}\label{OD1} For any integer $k$, $OD_{2k}$ divides $\Od[(4k-1)!]$.
\end{lemma}
\begin{proof}
The odd part of the denominator of the divided Bernoulli number is
$$OD_{2k}=\prod_{\substack{p-1\,|\, 2k\\ 
p \ \text {odd prime}}}p^{\mu+1},$$
where $p^{\mu}$ is the highest power of $p$ that divides $2k$. Consider the factor $p^{\mu+1}$ for some odd prime $p$ such that $p-1$ divides $2k$. If  $p \leq 2k-1$ and $p\nmid k$, then $p^{\mu+1}$ equals $p$ and divides $(2k-1)!$. If $p = 2k+1$, then $p^{\mu+1}$ equals $p$ and divides $(2k+1)$. Finally, if $p \mid k$, then $\mu \geq 1$ and $p^{\mu+1}$ divides $k^2$, which means it divides $(2k)(3k)$. Altogether, we have that $OD_{2k}$ divides $(2k-1)!(2k+1)(2k)(3k)$, which divides $(4k-1)!$. Since $OD_{2k}$ is odd, the result follows.\\
\end{proof}

Altogether, these arguments show that a $\QP^2$ exists in dimension $8k$ if and only if there exist integers $\bar{x},l \in \Z$ such that $\bar{x} \equiv l \bmod{2}$ and such that Equations \eqref{signature-eq1} and \eqref{e1L-eq3} hold. Substituting Equation \eqref{e1L-eq3} into Equation \eqref{signature-eq1}, we derive an equation in $\bar{x}$ and $l$ that holds for some $\bar{x}$ and $l$ of the same parity if and only if a $\QP^2$ exists in dimension $8k$. We determine the precise equations in the cases $k = 2^a$ and $k = 2^a + 2^b$ with $a<b$ separately.

 \begin{theorem}[Dimension $8k$ where $k = 2^a$]\label{case1}
There exists a $\QP^2$ in dimension $8k=8(2^a)$ if and only if there is integer solution $\bar{x}$ to the quadratic residue equation
\begin{equation}
a_k\bar{x}^2\equiv c_k\pmod{b_k} \label{iffa2}
\end{equation}
where
	\begin{eqnarray*}
	a_k	&=&	(2^{2k-1}-1)N_{2k}\left[\rho_k(2^{2k-1}-1)N_{2k}-2^{2k}(2^{4k-1}-1)N_{4k}\right],\\
	b_k	&=&	(2^{4k-1}-1)N_{4k}OD_{2k},\\
	c_k	&=&	2 OD_{2k} OD_{4k},
	\end{eqnarray*}
and where $\rho_k = OD_{4k} / OD_{2k}$. 
\end{theorem}

\begin{remark}
We remark that, if $\bar{x}$ is a solution to Equation \eqref{iffa2} and $l \in \Z$ such that $a_k \bar{x}^2 + b_k l = c_k$, then it follows by the parities of $a_k$, $b_k$, and $c_k$ that $\bar{x}$ and $l$ have the same parity. Hence the condition that $\bar{x} \equiv l \bmod{2}$ is not required in Theorem \ref{case1}.
\end{remark}

\begin{remark}
In this case of $k=2^a$, 
\[ OD_{2k} = \prod\limits_{\substack{p-1\,|\, 2k\\ 
p \ \text {odd prime}}}p^{\mu+1}=\prod\limits_{\substack{p-1=2^c\\ 
c\leq a+1}}p=\prod\limits_{\substack{F_i \text{ is a Fermat prime }\\F_i\leq 2^{a+1}+1}}F_i.\]
It follows that $\rho_k = OD_{4k} / OD_{2k}$ is $1$ unless $p = 4k+1$ is a Fermat prime, in which case $\rho_k = 4k+1$. The only known examples of Fermat primes are $F_i=2^{2^i}+1$ where $0 \leq i \leq 4$. It is known that $F_i$ is composite for $5\leq i \leq 32$.
\end{remark}

\begin{proof}
Since $k=2^a$, we have $\nu_2(k)=a$, $\wt(k)=1$, and $(-1)^{k+1}=-1$, so the signature equation \eqref{signature-eq1} becomes
\begin{eqnarray}
\left[\df{(2^{2k-1}-1)N_{2k}}{OD_{2k}} \bar{x}\right]^2+\df{(2^{4k-1}-1)N_{4k}}{OD_{4k}}\df{z}{\Od[(4k-1)!]}=2. \label{signature-eq2-case1}
\end{eqnarray}
The $e_1\mathord{\cdot} \LP$ condition \eqref{e1L-eq3} becomes
\begin{equation}z=\Od[(4k-1)!]\left(\frac{-2^{2k}\,(2^{2k-1}-1)N_{2k}}{OD_{2k}}\,\bar{x}^2+l\right).\label{e1L-eq4-case1}
\end{equation}
Substituting Equation \eqref{e1L-eq4-case1} into Equation \eqref{signature-eq2-case1}, replacing $OD_{4k}$ by $\rho_k OD_{2k}$, and simplifying yields 
	\[a_k \bar{x}^2 + b_k l = c_k,\]
where $a_k$, $b_k$, and $c_k$ are as in the theorem. Reducing modulo $b_k$, we obtain Congruence \eqref{iffa2}. 
\end{proof}

We now consider dimensions of the form $n = 8k = 8(2^a + 2^b)$ with $a < b$. Recall that it remains an open problem whether such a dimension supports a $\QP^2$.

\begin{theorem}[Dimensions $8k$ where $k = 2^a + 2^b$ and $a < b$]\label{case2}
There exists a $\QP^2$ in dimension $8k=8(2^a+2^b)$ with $a\not=b$ if and only if there is an odd integer solution $\bar{x}$ to the quadratic residue equation \begin{equation}
A_k\bar{x}^2\equiv C_k\pmod{B_k} \label{iffa3}
\end{equation}
where
	\begin{eqnarray*}
	A_k	&=&	2(2^{2k-1}-1)N_{2k}\left[(2^{2k-1}-1)N_{2k}\left(\frac{OD_{4k}}{OD_{2k}}\right)+ (-1)^{k+1}2^{2k}(2^{4k-1}-1)N_{4k}\right],\\
	B_k	&=&	(2^{4k-1}-1)N_{4k}OD_{2k},\\
	C_k	&=&	OD_{2k} OD_{4k}.
	\end{eqnarray*}
\end{theorem}

\begin{proof}
In the case that $k=2^a+2^b$ with $a\not = b$, we have $\wt(k)=2$, and $\nu_2(k)=\min\{a,b\}=a$ without loss of generality, so the signature equation \eqref{signature-eq1} becomes
\begin{eqnarray}
2\left[\df{(2^{2k-1}-1)N_{2k}}{OD_{2k}} \bar{x}\right]^2+\df{(2^{4k-1}-1)N_{4k}}{OD_{4k}}\df{z}{\Od[(4k-1)!]}=1.\label{signature-eq2-case2}
\end{eqnarray}
The $e_1\mathord{\cdot} \LP$ condition  \eqref{e1L-eq3} becomes
\begin{equation}
z=\Od[(4k-1)!]\left[\frac{(-1)^{k+1}2^{2k+1}(2^{2k-1}-1)N_{2k}}{OD_{2k}}\bar{x}^2+l\right]. \label{e1L-eq4-case2}
\end{equation}
Substituting  \eqref{e1L-eq4-case2} into \eqref{signature-eq2-case2} and proceeding as in the previous proof implies the theorem.
\end{proof}

\section{Existence in dimensions 128 and 256}\label{sec:Existence}

Recall that dimensions $4$, $8$, $16$, and $32$ are known to support the existence of a $\QP^2$. Having simplified the signature and Hattori--Stong integrality conditions to a single quadratic reciprocity condition in the previous section, we proceed to the proof that dimensions $128$ and $256$ also support a $\QP^2$.

 \begin{proof}[Proof of existence in dimensions $128$ and $256$]
It suffices to prove that Equation \eqref{iffa2} has solution when $8k = 128$, i.e., when $k = 16$. Factoring out the common divisor of $OD_{2k} = 3\mathord{\cdot}5\mathord{\cdot}17$ from $a_k$, $b_k$, and $c_k$, Equation \eqref{iffa2} is equivalent to an equation of the form
$$a\bar{x}^2\equiv c\pmod{b}$$
where $\gcd(a,b) = 1$ and $\gcd(c,b) = 1$.  The coefficients are large, so we do not include the calculations here. It suffices to solve the equation $\bar{x}^2 \equiv a^{-1} c \pmod b$. Now $a^{-1} c$ is a quadratic residue modulo $b$ if and only if it is a quadratic residue modulo all odd prime factors $p$ of $b$. Hence it suffices to show that the Legendre symbols $\left(\frac{a}{p}\right)=\left(\frac{c}{p}\right)$ for all prime factors $p$ of $b$, and one can easily verify this using Mathematica.

For dimension $256$, one proceeds similarly to check that Equation \ref{iffa2} has a solution when $8k = 256$, i.e., when $k = 32$. Again it happens that the greatest common divisor of $a_k$, $b_k$, and $c_k$ is $OD_{2k} = 3\mathord\cdot 5 \mathord\cdot 17$.
\end{proof}

\section{Non-existence results in higher dimensions}\label{sec:Non-existence}

So far, all the dimensions known to not support a $\QP^2$ were proved by obstructing the signature equation. As stated in \cite[Lemma 3.2]{FowlerSu16}, one can search for an irregular prime $p$ such that $p\equiv5\pmod{8}$, $\nu_p(s_{2k})>0$ and $\nu_p(s_k)=0$ to obstruct the signature equation in a candidate dimension of the form $n=8k$ where $k=2^a+2^b$. Adopting the same idea and using the more explicit necessary and sufficient conditions derived in Theorems \ref{case1} and \ref{case2}, we prove the following proposition stating that any prime $p \equiv \pm 3 \pmod{8}$ detected as a factor of the the numerator of the divided Bernoulli number is an ``obstructing" prime.

\begin{proposition}\label{N-prime}
If the numerator $N_{4k}$ of $\frac{|B_{4k}|}{4k}$ has a prime factor $p \equiv \pm 3 \pmod{8}$, then there does not exist a $\QP^2$ in dimension $n=8k$.  In particular, if $N_{4k} \equiv \pm 3\pmod 8$, then there is no $\QP^2$ in dimension $8k$. 
\end{proposition}

\begin{proof}
The second statement follows immediately from the first. To prove the first, we claim that a $\QP^2$ exists in dimension $8k$ only if two is a quadratic residue modulo $N_{4k}$. Indeed, when $k=2^a$, Theorem \ref{case1} implies that some $\bar{x} \in \Z$ exists such that
\begin{eqnarray}\label{case1-necessary}
\left[\rho_k(2^{2k-1}-1)N_{2k}\bar{x}\right]^2\equiv 2OD_{4k}^2\pmod{N_{4k}}
\end{eqnarray}
Similarly, when $k = 2^a+2^b$ and $a\neq b$, Theorem \ref{case2} implies that
\begin{eqnarray}\label{case2-necessary}
2\left[\left(\df{OD_{4k}}{OD_{2k}}\right)(2^{2k-1}-1)N_{2k}\bar{x}\right]^2\equiv OD_{4k}^2\pmod{N_{4k}}
\end{eqnarray}
Since $OD_{4k}^2$ and $N_{4k}$ are coprime, the claim follows.

Now if $N_{4k}$ has a prime factor $p \equiv \pm 3 \pmod{8}$, 2 is a quadratic nonresidue modulo $p$. Since two is a quadratic residue modulo $N_{4k}$ only if two is a quadratic residue modulo $p^r$ for every prime power dividing $N_{4k}$, 2 is also a quadratic nonresidue modulo $N_{4k}$. This implies that no $\QP^2$ exists in this dimension.
\end{proof}

In the following corollary, we use Carlitz's congruence to find families of dimensions where $N_{4k} \equiv \pm 3 \pmod 8$. Then Proposition \ref{N-prime} implies non-existence of $\QP^2$ in these dimensions. 

\begin{corollary}\label{N-mod8}
No $\QP^2$ exists in dimension $8k$ for all $k$ of the form $2^{a+i} + 2^{a}$ with $i \in \{1,2,3,5,7\}$ and $a \geq 0$. 
\end{corollary}

Note that the corollary provides infinite families of dimensions $8(2^a+2^b)$ with $a\not= b$ that do not support a $\QP^2$, which implies part of Theorem \ref{thm:InfinitelyManyPowersOfTwo}. We remark that this corollary holds for many more values of $i$, and we suspect it holds for infinitely many values of $i$.

\begin{proof}
We show that $N_{4k} \equiv \pm 3 \pmod 8$ for all $k$ of the form $2^{a+i} + 2^{a}$ with $i \in \{1,2,3,5,7\}$ and $a \geq 0$.
Firstly one can computationally verify the claimed values $k$ of the form $2^i + 1$ (i.e., those special values with $a = 0$). This can be done with a computer or by hand using some of the observations that follow. We omit the proof of this part. Once this is done, it suffices to show that $N_{4k} \equiv N_{4(2^i+1)}$ for all $k$ of the form $2^{a+i} + 2^a = 2^a(2^i + 1)$. To show the latter claim, recall that Carlitz \cite{Carlitz53} proved that $2^{a+3}$ divides $2B_{4k} - 1$ since $2^{a+2}$ divides $4k$ (cf. \cite[Theorem 2]{Howard95}). We write $B_{4k} = 4 k N_{4k} / D_{4k} = \Od[4k] N_{4k} / (2 OD_{4k})$ in terms of the numerator $N_{4k}$ and denominator $D_{4k} = 2^{\nu_2(4k)+1} OD_{4k}$ of the divided Bernoulli number $B_{4k}/(4k)$. Multiplying by $2 OD_{4k}$ and applying the Carlitz congruence, we have that $(2^i + 1) N_{4k} \equiv OD_{4k}$ modulo $2^{a+3}$ and hence modulo $8$. To complete the proof, it suffices to show that the reduction of $OD_{4k}$ modulo $8$ is independent of $a$ where again $k = 2^a(2^i + 1)$.

We have that $OD_{4k}$ is the product of $p^{1 + \nu_p(4k)}$ over odd primes $p$ such that $p-1 | 4k$. Note that $\nu_p(4k) = \nu_p(2^i+1)$ for odd primes $p$. Note also that $p \neq 2$ and $p-1 | 4k$ implies that $p = 2^c d + 1$ for some $1 \leq c \leq a + 2$ and some divisor $d$ of $2^i+1$. Note moreover that $c \geq 3$ implies that $p \equiv 1 \pmod{8}$. Hence 
	\[OD_{4k} \equiv  \prod_{p \in P_2 \cup P_4} p^{1 + \nu_p(2^i + 1)} \pmod{8}\]
where $P_2$ is the set of primes $p$ of the form $2d+1$ for some divisor $d$ of $2^i + 1$ and where, similarly, $P_4$ is the set of primes $p$ of the form $4d+1$ for some divisor $d$ of $2^i+1$. Clearly this quantity is independent of $a$, so we have $N_{4\cdot 2^a(2^i+ 1)} \equiv N_{4(2^i+1)} \pmod{8}$, as claimed.
\end{proof}

Note that the problem in dimensions less than $256$ has been resolved in \cite{Su14,FowlerSu16}. Now we are ready to prove the non-existence dimensions included in Theorem \ref{thm:UpTo512}.

\begin{theorem}[Theorem A]
There does not exist a $\QP^2$ in dimension $n=8k$ when $256<n<2^{13}$ except possibly when $n \in \{544, 1024, 2048, 4160, 4352\}$.
\end{theorem}

\begin{proof}
For all $8k=8(2^a+2^b)$ strictly between $256=8(2^5)$ and $8192=8(2^{10})$ except the five exceptions stated in the theorem, we show that the numerator of the divided Bernoulli number $N_{4k}$ either is congruent to $\pm 3\pmod{8}$ itself, or it has a prime divisor $p \equiv \pm 3\pmod{8}$. Then Proposition \ref{N-prime} concludes these dimensions do not support a $\QP^2$. 

Firstly, we eliminate all dimensions $8k$ where $k=2^a + 2^{a-i}$ with $i \in \{1,2,3,5,7\}$, since we have shown in Corollary \ref{N-mod8} that $N_{4k}$ itself is congruent to $\pm 3 \pmod 8$ in these dimensions. In Table \ref{table-prime}, we list all the remaining values of $k=2^a + 2^b$ in the range we consider. While $N_{4k} \equiv \pm 1 \pmod 8$ in each of the dimensions, we frequently find $N_{4k}$ has an irregular prime factor $p \equiv \pm 3\pmod{8}$, which then obstruct existence of $\QP^2$ by Proposition \ref{N-prime}. 

Note that for the values of $k$ of the form $2^a$ and in the range we consider, we are able exclude $k = 2^6$ (i.e., dimension $2^9$) and $k = 2^9$ (i.e., dimension $2^{12}$) using the irregular primes $67$ and $37$, respectively.
\end{proof}

\begin{remark}
We remark on the limits of this method to further obstruct existence of $\QP^2$. The Bernoulli numerators and their irregular prime factors are of great importance in number theory, and with the aid of computers, factorizations of high order Bernoulli numerators have been done by various authors. Sam Wagstaff's webpage \cite{Wagstaff13} maintains a list of known prime factors of the Bernoulli numerators up to $B_{300}$. We used this list to check whether $N_{4k}$ has a prime factor $p\equiv \pm 3\pmod 8$.

In dimensions $8k \in \{544, 1024, 2048, 4160, 4352, 8192\}$, we put ``?"  in the column of irregular prime. This indicates that, based on \cite{Wagstaff13}, we do not know whether $N_{4k}$ has a prime factor $p\equiv \pm 3\pmod{8}$. 
\end{remark}

\begin{table}[h]
  \caption{Dimensions up to $2^{13}$ of the form $8k = 8(2^a + 2^b)$ with $a > b$ that are not ruled out by Proposition \ref{N-prime}.
  }
  \label{table-prime}
\begin{tabular}{c|c|r|c}
\shortstack{$(a,b)$\\ \ } & \shortstack{prime factor $p|N_{4k}$\\  
 with $p\equiv \pm 3 \pmod{8}$} &  \shortstack{dimension \\
 $n=8k$} &
   \shortstack{there exists a $\QP^2$ \\
   in dimension $n$ ?} \\
   \hline
   \hline
%$(5,0)$ & $132$ & Y & Y &  $n=264$, No\\ \hline
$(5,1)$ & $29835096585483934621$ & $272$ & No\\\hline
%$(5,2)$ & $144$ & Y & Y & $288$, No\\\hline
%  $(5,3)$ & $160$ & Y & Y  &  $320$, No\\\hline
%   $(5,4)$ & $192$ & Y & Y  &  $384$, No\\ \hline
   $(5,5)$ & $67$ &   $8(2^6)=512$ &No\\  \hline
   $(6,0)$ & $15897346573$ &  $520$& No\\\hline
%    $(6,1)$ & $264$ & Y & Y &  $528$, No\\\hline
  $(6,2)$ & $?$ & $544$ & ? \\\hline
%   $(6,3)$ & $288$ & Y & Y  &  $576$, No\\\hline
%   $(6,4)$ & $320$ & Y & Y &  $640$, No\\\hline
%   $(6,5)$ & $384$ & Y & Y &  $768$, No\\\hline
   $(6,6)$ & $?$ & $8(2^7)=1024$ & $?$\\\hline
 %  $(7,0)$ & $516$ & Y & Y  &  $1032$, No\\\hline
   $(7,1)$ &$67$ & $1040$ & No\\\hline
 %  $(7,2)$ & $528$ & Y & Y   &  $1056$, No\\\hline
   $(7,3)$  & $811$ &  $1088$ & No\\\hline
 %  $(7,4)$ & $576$ & Y & Y  &  $1152$, No\\\hline
 %  $(7,5)$ & $640$ & Y & Y  &  $1280$, No\\\hline
 %  $(7,6)$ & $768$ & Y & Y   &  $1536$, No\\\hline
   $(7,7)$ &  $?$ & $8(2^8)=2048$ & $?$\\\hline
   $(8,0)$  & $26251$ & $2056$ & No\\\hline
  % $(8,1)$ & $1032$ & Y & Y &  $2064$, No\\   \hline
   $(8,2)$  & 37 & $2080$ &  No\\   \hline
  % $(8,3)$ & $1056$ & Y & Y &  $2112$, No\\   \hline
   $(8,4)$  & 59 &  $2176$ & No\\   \hline
  % $(8,5)$ & $1152$ & Y & Y &  $2304$, No\\\hline
  % $(8,6)$ & $1280$ & Y & Y &  $2560$, No\\  \hline
  % $(8,7)$ & $1536$ & Y & Y &  $3072$, No\\  \hline
  $(8,8)$ & 37 & $8(2^9)=4096$ & No\\\hline
   $(9,0)$  & 4349 & $4104$ &  No\\\hline
   $(9,1)$  & 1669 & $4112$ & No\\ \hline
  % $(9,2)$ & $2064$ & Y & Y &  $4128$, No\\ \hline
   $(9,3)$  & $?$ & $4160$ & $?$ \\  \hline
  % $(9,4)$ & $2112$ & Y & Y &  $4224$, No\\  \hline
   $(9,5)$  & $?$ & $4352$ & $?$ \\  \hline
  % $(9,6)$ & $2304$ & Y & Y &  $4608$, No\\  \hline
  % $(9,7)$ & $2560$ & Y & Y &  $5120$, No\\  \hline
  % $(9,8)$ & $3072$ & Y & Y &  $6144$, No\\              
%\hline
   $(9,9)$ &$?$ &  $8(2^{10})=8192$ & $?$\\
\end{tabular}
\end{table}

We now state a second approach to obtain more nonexistence results. We thank Sam Wagstaff for pointing us to the Kummer's congruence, which is applied to extend Proposition \ref{N-prime} to rule out families of dimensions by the obstructing irregular primes.  
 
\begin{proposition}\label{N-Kummer}
If the numerator $N_m$ of $\frac{|B_{m}|}{m}$ has a prime factor $p \equiv \pm 3 \pmod{8}$,
then for any $k$ such that $4k\equiv m\pmod{p-1}$, there does not exist a $\QP^2$ in dimension $8k$.  
\end{proposition}

\begin{proof}
Suppose $p$ is an prime factor of the numerator of $\frac{|B_{m}|}{m}$. By Kummer's congruence, whenever $4k\equiv m\pmod{p-1}$,
$$\df{B_{4k}}{4k}\equiv\df{B_{m}}{m}\equiv 0\pmod{p},$$
so $p$ is also a prime factor of the numerator of $\frac{|B_{4k}|}{4k}$. If, in addition, $p\equiv \pm 3\pmod{8}$, Proposition \ref{N-prime} implies that no $\QP^2$ exists in dimension $8k$.
\end{proof}

Applying Proposition \ref{N-Kummer} to the first irregular prime $37$, which divides $N_{32}$, we obtain the following result.

\begin{proposition}[Obstruction by the irregular prime $37$ dividing $N_{32}$]\label{kummer-37}
There does not exist a $\QP^2$ in any dimension of the form $n=2^{6r+5}+2^{6s+5}$ or $n=2^{6r+3} + 2^{6s+7}$ for any $r,s\in \Z_{\geq 0}$. In particular, there is no $\QP^2$ in dimension $2^{6r}$ for any $r \geq 1$.
\end{proposition}
\begin{proof}

Note that $4k=4(2^a+2^b)\equiv 32 \pmod{37-1}$ whenever $2^a+2^b\equiv 8\pmod{9}$. This holds whenever $(a, b)\equiv (2, 2) \pmod{6}$ or $(a,b) \equiv (0, 4)\pmod{6}$, as $2^6\equiv 1\pmod{9}$ by the Euler's theorem. Then by Proposition  \ref{N-Kummer}, these two cases correspond to the dimensions $n=8k$ stated in the theorem. 
\end{proof}

We apply Proposition \ref{N-Kummer} to the first thirteen irregular primes congruent to $\pm 3\pmod{8}$, the results are listed in Table \ref{table-kummer}. The following proposition summarizes the families of nonexistence dimensions of the form $n=2^a$ obstructed by the primes $p \in \{37, 67, 101, 59, 389, 347\}$. Together with Corollary \ref{N-mod8}, this completes the proof of Theorem \ref{thm:InfinitelyManyPowersOfTwo}. 

\begin{proposition}[Obstruction to dimensions $2^a$ by the first few irregular primes]\label{kummer-case1}
There does not exist a $\QP^2$ in any dimension of the form $2^{6r+6}$, $2^{10r+9}$, $2^{20r+16}$, $2^{28r+28}$, $2^{48r+21}$, or $2^{172r+138}$ for any $r\in\Z_{\geq 0}$.
\end{proposition}

\begin{remark}
Proposition \ref{N-Kummer} provides new infinite families of dimensions that do not support a $\QP^2$. Moreover, it seems that most irregular primes, of which there exist infinitely many, provide such families of dimensions that do not support a $\QP^2$. It seems to be a difficult problem to classify the dimensions that are obstructed by such arguments. 
\end{remark}

\begin{scriptsize}
\begin{table}[h]
  \caption{Dimensions ruled out by Proposition \ref{N-Kummer}}
  \label{table-kummer}
\begin{tabular}{c|c|c}
\shortstack{irregular prime $p\,|\,N_{m}$,\\  
 $p\equiv\pm3\pmod{8}$} & \shortstack{$(a,b)$ such that \\ $4k=4(2^a+2^b)\equiv m\pmod{p-1}$} &
   \shortstack{dim $n=8k$ ($n>256$) that  \\ does not support a $\QP^2$
   } \\
   \hline
   \hline
\shortstack{$37\,|\,N_{32}$\\ \ \ } & \shortstack{$(a,b)\equiv (2,2)$;$(0,4)\pmod{6}$\\ \ \ } &
\shortstack{$2^{6r+5}+2^{6s+5}; $\\
$2^{6r+3}+2^{6s+7}$}\\ \hline
\shortstack{$59\,|\,N_{44}$\\ \ \ \\\ \ \\ \ \ \\ \ \ \\ \ } & \shortstack{$(24,24);(0,23); (1,10); (2,12); \ \ \ \ \  \ \ \ \ \  \ \ \ \ \ $\\$(3,5); (4,8); (6,22);(7,14); $ \\ $(9,17); (11,26);(13,19); (15,18); $ \\ $(16,27); (20,21)\pmod{28}$} &
\shortstack{$2^{28r+27}+2^{28s+27};$\\
$2^{28r+3}+2^{28s+26};$\\
$\ldots$}\\ \hline
\shortstack{$67\,|\,N_{58}$\\ \ \ } & \shortstack{$(5,5)$;$(1,7)\pmod{10}$\\ \ \ } &
\shortstack{$2^{10r+8}+2^{10s+8};$\\
$2^{10r+4}+2^{10s+10}$\\
$\ldots$}\\ \hline
\shortstack{$101\,|\,N_{68}$\\ \ \ \\\ \ } & \shortstack{$(12,12);(0,4); (2,19);  \ \ \ \  \ \ \ \ \  \ \ \ \ \ $\\$ (3,14);(6,7); (8,16);  \ \ \ \  \ \ \ \ $\\$ (10,15);(11,18)\pmod{20}$} &
\shortstack{$2^{20r+15}+2^{20s+15};$\\
$2^{20r+3}+2^{20s+7};$\\
$\ldots$}\\ \hline
$131\,|\,N_{22}$ & no such $(a,b)$ &
\ \\ \hline
$149\,|\,N_{130}$ & no such $(a,b)$ &
\ \\ \hline
$157\,|\,N_{62}$ and $N_{110}$ & no such $(a,b)$ &
\ \\ \hline
\shortstack{$283\,|\,N_{20}$\\ \ \ \\\ \ } & \shortstack{$(0,2); (4,40); (14,22);  \ \ \ \ \  \ \ \ \ \  \ \ \ \ \ $\\$ (16,30);(18,34);(24,42)\pmod{46}$} &
\shortstack{$2^{46r+3}+2^{46s+5};$\\
$2^{46r+7}+2^{46s+43};$\\
$\ldots$}\\ \hline
\shortstack{$293\,|\,N_{156}$\\ \ \ } & \shortstack{$ (1,8)\pmod{9}$\\ \ \ } &
\shortstack{$2^{9r+4}+2^{9s+11}$\\
\ \ \\
\ \ }\\ \hline
$307\,|\,N_{88}$ & no such $(a,b)$ &
\ \\ \hline
\shortstack{$347\,|\,N_{280}$\\ \ \ } & \shortstack{$(134,134);(0, 47);\ \ \ \ \  \ \ \ \ \  \ \ \ \ \ $\\$ (26, 141);\cdots
\pmod{172}$\\ \ \ } &
\shortstack{$2^{172r+137}+2^{172s+137};$\\
$2^{172r+3}+2^{172s+50};$\\
$\ldots$}\\ \hline
\shortstack{$379\,|\,N_{174}$\\ \ \ } & \shortstack{$(2,9)$;$(3,14);(8,15)\pmod{18}$\\ \ \ } &
\shortstack{$2^{18r+5}+2^{18s+12};$\\
$2^{18r+6}+2^{18s+17};$\\
$\ldots$}\\ \hline
\shortstack{$389\,|\,N_{200}$\\ \ \ \\\ \ \\ \ \ \\ \ \ \\ \ } & \shortstack{$(17,17);(1,23); (3,39); (8,45); \ \ \ \ \  \ \ \ \ \  \ \ \ \ \ $\\$(10,26); (13,20); (16,35);(19,42); $ \\ $(28,31); (36,41)\pmod{48}$} &
\shortstack{$2^{48r+20}+2^{48s+20};$\\
$2^{48r+4}+2^{48s+26};$\\
$\ldots$}\\ \hline
\end{tabular}
\end{table}
\end{scriptsize}

\section{Spin $\QP^2$}\label{sec:Spin}

As studied in \cite{FowlerSu16}, if a smooth manifold $M$ is a $\QP^2$ that admits a spin structure, the following conditions must hold true:
\begin{enumerate}[(1')]
\item
(Hirzebruch signature equation)\ \ $\langle \LP(p_k,p_{2k}),\mu\rangle=s_{k,k}\langle p_k^2,\mu\rangle+s_{2k}\langle p_{2k},\mu\rangle=1$,
\item
(Stong integrality condition from $\Omega^{Spin}_{8k}$)\
\begin{equation}\label{spin-stong}\langle\Z[e_1,\,e_2, \ldots]\mathord{\cdot}\AP\,,\,\mu\rangle\in\Z\end{equation}
\item
(Pontryagin numbers of $\QP^2$)
\begin{equation}\langle p_k^2,\mu\rangle= x^2 \mbox{ and } \langle p_{2k},\mu\rangle=y \ \ \mbox{ for some integers } x \mbox{ and } y\nonumber\end{equation}
\end{enumerate}
Condition (2') characterizes the integral lattice in $\Q^{p(8k)}$ formed by all possible Pontryagin numbers of smooth $8k$-dim Spin manifolds. The total $\AP$ class can be written as
$$\AP=1+a_kp_k+a_{k,k}p_k^2+a_{2k}p_{2k},$$
where the coefficients 
\begin{eqnarray*}
a_k&=&\df{-|B_{2k}|}{2(2k)!},\\
a_{k,k}&=&\frac{1}{2}(a_{k}^2-a_{2k}).
\end{eqnarray*}

Similar to the smooth case, the signature equation (1') and the spin integrality condition (2') together can be written as a set of integrality conditions on the Pontryagin numbers $x^2=\langle p_k^2,\mu\rangle$ and $y=\langle p_{2k},\mu\rangle$. In \cite{FowlerSu16}, it was shown that there is no solution to (1') and (2') together in dimension 32, which proved the nonexistence of spin structure on any 32 dimensional $\QP^2$. Now we prove the following theorem, a special case of which asserts the nonexistence of Spin $\QP^2$ in any dimension greater than $16$. 

\begin{theorem}\label{thm:SpinPLUS}
Let $M^{8k}$ be a simply connected closed smooth manifold that admits a spin structure. Assume all Pontryagin numbers of $M$ vanish except possibly for $\xi = p_{k}^2[M]$ and $y = p_{2k}[M]$. If the signature $\sigma = \sigma(M)$ is nonzero, then
	\[ \nu_2(2\sigma) \geq 4k - 2\nu_2(k) - 5.\]
\end{theorem}

In our case of Spin $\QP^2$, the dimension is either four or of the form $8k$. Since a 4--dimensional Spin manifold must have even intersection form, a $\QP^2$ in dimension four cannot be Spin. For dimensions $8k$, Theorem \ref{thm:SpinPLUS} applies. Since the signature is $1$, we have the estimate
	\[1 \geq 4k - 2\nu_2(k) - 5 \geq 4k - 2\log_2(k) - 5,\]
which is contradiction unless $k \in \{1,2\}$. Hence the following is immediate.

\begin{corollary}[Theorem \ref{thm:Spin}]
A $\QP^2$ admitting a Spin structure can only exist in dimensions $8$ and $16$, i.e., the dimensions of $\HP^2$ and $\OP^2$.
\end{corollary}

\begin{proof}[Proof of Theorem \ref{thm:SpinPLUS}]
Assume that such a manifold exists. Its Pontryagin numbers $\xi = p_{k}^2[M]$ and $y = p_{2k}[M]$ satisfy the signature equation, the $\AP$ genus condition, and the $e_1e_1\mathord{\cdot} \AP$ condition. Hence
\begin{subnumcases}
\ \langle \LP,\mu\rangle=s_{k,k} \xi +s_{2k}y=\sigma. \label{L}\\
\langle \AP,\mu\rangle=a_{k,k} \xi +a_{2k}y\in\Z. \label{A}
\\
\langle e_1e_1\mathord{\cdot}\AP,\mu\rangle= \df{\xi}{[(2k-1)!]^2}\in\Z.\label{e1e1A}
\end{subnumcases}
By \eqref{e1e1A}, $\xi = [(2k-1)!]^2 \xi_1$ for some integer $\xi_1$. Let $z=2y-\xi$. The signature equation \eqref{L} and the $\AP$ genus condition \eqref{A} can be written as
\begin{eqnarray*}
[s_k (2k-1)!]^2 \xi_1 + s_{2k}z			&=&		2\sigma \\
\left[ a_k(2k-1)!\right]^2 \xi_1 + a_{2k} z 	&=&		2m
\end{eqnarray*}
for some $m \in \Z$. Using the fact that $s_{2k} = - 2^{4k+1}(2^{4k-1} - 1) a_{2k}$, we use the second equation to eliminate $z$ in the first equation. This yields, after simplification,
	\[2^{4k+1}\left[(2k-1)! a_k\right]^2 (2^{2k} - 1)^2 \xi_1 - 2^{4k+2} (2^{4k-1} - 1) m = 2\sigma.\]
Computing $\nu_2$ of each of the two summands on the left-hand side yields 
	\[4k - 5 - 2\nu_2(k) + \nu_2(\xi_1)\]
and $4k+2 + \nu_2(m)$. Both of these are at least $4k - 5 - 2\nu_2(k)$, so
	\[\nu_2(2\sigma) \geq 4k - 5 - 2\nu_2(k),\]
as claimed.
\end{proof}

\section{Existence of rational projective spaces}
\label{sec: projective-space}

Generalizing the notion of rational projective plane, a simply connected closed smooth manifold $M$ is called a rational projective space if $H^*(M;\Q) \cong \Q[\alpha]/\langle \alpha^{n+1}\rangle, n\geq1$. We let $\QP^n_d$ denote a $(nd)$--dimensional rational projective space where $d$ is the degree of the generator. In \cite{FowlerSu16}, it was shown that higher dimensional analogues of rational Cayley planes, i.e., $\QP^n_8$ for $n > 2$, exist in dimension $8n$ whenever $n$ is odd. We prove the following theorem that extends existence results on rational projective plane to rational projective spaces. 

\begin{nonumbertheorem}[Theorem \ref{thm:QPn}]
If a $\QP^2_{4k}$ exists, then a $\QP^{2m}_{4k/m}$ exists whenever $4k/m \in 2\Z$.
\end{nonumbertheorem}

\begin{proof} 

Assume $m$ is an integer such that $4k/m$ is an even integer. Let $\A$ denote the $8k$-dimensional rational graded commutative algebra $\Q[\alpha]/\langle \alpha^{2m+1}\rangle$ where $|\alpha|=4k/m$. Note that $\A$ is realizable as a cohomology ring only if the degree of the generator $|\alpha|=4k/m$ is even. By the Sullivan-Barge rational surgery realization theorem, there exists a $8k$-dimensional closed smooth manifold $M$ such that $H^*(M;\Q)=\A$ if and only if there exist choices of cohomology classes $p_i\in H^{4i}(X;\Q)$ for $i=1,\ldots 2k$, where $X$ is a $\Q$-local space carrying the desired rational cohomology data such that $ H^*(X;\Q)= H^*(X;\Z)=\A$; and a choice of fundamental class $\mu\in H_{8k}(X;\Q)$, such that the pairs $\langle p_{i_1}\cdots p_{i_r},\mu\rangle, i_1+\cdots +i_r=2k$ are integers that satisfy 

\begin{enumerate}
\item[(i)] The signature equation that $\langle \LP(p_1,\ldots,p_{2k}),\mu\rangle=1$
\item[(ii)] The Hattori--Stong integrality conditions that $\langle\Z[e_1,\,e_2, \ldots]\mathord{\mathord{\cdot}}\LP\,,\,\mu\rangle\in\Z[1/2]$. 
\item[(iii)] The rational intersection form $\langle \cdot \cup \cdot, \mu\rangle$ is isomorphic to $\langle 1\rangle$.
\end{enumerate}

If we let the choice of cohomology classes be $p_i=0$ except $p_k$ and $p_k^2$, Conditions (i) and (ii) become exactly the same as the corresponding conditions to realize a $\QP^2$, which are stated as (1) and (2) in section 1. Moreover, the substitution stated in (3) in the $\QP^2$ case still holds true here. By the desired rational cohomology ring $\A$, any choice of cohomology classes $p_k$ and $p_{2k}$ can be written as $p_k=a\alpha^m$ and $p_{2k}=b\alpha^{2m}$ for some rational numbers $a$ and $b$. Under a choice of orientation, (iii) requires the rational intersection form with respect to $\mu$ to be isomorphic to $\langle 1 \rangle$ and the signature is $1$, so the choice of $\mu$ must satisfy  $\langle \alpha^{2m}, \mu\rangle= r^2$ for some rational number $r$, therefore we may still express the pairs $\langle p_k^2,\mu\rangle=a^2r^2= x^2$ and $\langle p_{2k},\mu\rangle=br^2= y$, where $x$ and $y$ must be integers. So under such choice of having all $p_i=0$ except $p_k$ and $p_k^2$, the sufficient conditions to realize a $\QP^2_{4k}$ in dimension $8k$ are also the sufficient conditions to realize a $\QP^{2m}_{4k/m}$ in dimension $8k$.
\end{proof}

As an application of this theorem, combined with the the existence of a $\QP^2$ in dimensions $32$, $128$, and $256$, we have the following existence results of rational projective spaces.

\begin{corollary}
Each of the following manifolds exists.
	\begin{enumerate}[{I.}]
	\item Higher dimensional analogues, $\QP_8^n$ for $n \in \{4,16,32\}$, of rational Cayley planes.
	\item Higher dimensional analogues $\QP_{16}^8$ and $\QP_{16}^{16}$ of the $32$--dimensional $\QP^2$.
	\item Manifolds $\QP_{32}^4$ and $\QP_{32}^8$, despite the fact that no rational projective plane exists with generator in degree $32$.
	\item Manifold $\QP_{64}^4$.
	\end{enumerate}
\end{corollary}

\begin{remark}
Note that a dimension not supporting a $\QP^2_{4k}$ is not necessarily one that does not support a $\QP^{2m}_{4k/m}$. The sufficient conditions (i), (ii), and (iii) in the proof above might be realized by choices of cohomology classes with nonzero $p_i$ other than $p_k$ and $p_{2k}$.  
\end{remark}

\begin{remark}
It is natural to ask if one can obtain a general existence theorem for rational projective spaces similar to the quadratic residue equation stated in Theorem \ref{case1} and Theorem \ref{case2} for rational projective planes. For the case of $\QP^4_{4k}$, which has rational cohomology ring $\Q[\alpha]/\langle \alpha^{5}\rangle,\  |\alpha|=4k$, as addressed in \cite[Remark 6.2]{FowlerSu16}, the signature equation becomes a quartic Diophantine equation with 4 unknowns if we assume each of the four Pontryagin classes $p_k$, $p_{2k}$, $p_{3k}$, and $p_{4k}$ could be nonzero. It also remains to be seen whether one can simplify the Hattori--Stong integrality conditions in this case.
\end{remark}

%%%%% Bibliography %%%%%
%\bibliographystyle{alpha}
%\bibliography{myrefs}

\end{document}